\documentclass[12pt]{article}

\usepackage{amsmath,amssymb,latexsym,url,graphicx,amsthm,multirow,color,enumerate}
\usepackage{pgf,tikz}
\usetikzlibrary{arrows}

\usepackage[colorlinks=true,citecolor=black,linkcolor=black,urlcolor=blue]{hyperref}

\newcommand{\arxiv}[1]{\href{http://arxiv.org/abs/#1}{\texttt{arXiv:#1}}}

\theoremstyle{plain}

\newtheorem{cor}{Corollary}

\newtheorem{lemma}{Lemma}

\newtheorem{prop}{Proposition}

\newcommand{\red}{\mathrm{red}}

\title{Ascent sequences avoiding pairs of patterns}

\author{Andrew M. Baxter \\
\small Mathematics Department \\[-0.8ex]
\small Pennsylvania State University \\[-0.8ex]
\small State College, PA, USA \\[-0.8ex]
\small \texttt{baxter@math.psu.edu}\\
\and
Lara K. Pudwell\\
\small Department of Mathematics and Statistics \\[-0.8ex]
\small Valparaiso University \\[-0.8ex]
\small Valparaiso, IN, USA \\[-0.8ex]
\small \texttt{Lara.Pudwell@valpo.edu}
}

\begin{document}

\maketitle

\abstract{
Ascent sequences were introduced by Bousquet-Melou et al. in connection with (2+2)-avoiding posets and their pattern avoidance properties were first considered by Duncan and Steingr\'{i}msson.  In this paper, we consider ascent sequences of length $n$ avoiding two patterns of length 3, and we determine an exact enumeration for 16 different pairs of patterns.    Methods include simple recurrences, bijections to other combinatorial objects (including Dyck paths and pattern-avoiding permutations), and generating trees.   We also provide an analogue of the Erd\H{o}s-Szekeres Theorem to prove that any sufficiently long ascent sequence contains either many copies of the same number or a long  increasing subsequence, with a precise bound.
}

\section{Introduction}\label{Introduction}

Given an integer string $x_1\dotsm x_n$, an \emph{ascent} is position $j$ such that $x_j<x_{j+1}$.  Write $\mathrm{asc}(x_1 \dotsm x_n)$ for the number of ascents in $x_1\dotsm x_n$.  An \emph{ascent sequence} $x_1\dotsm x_n$ is a sequence of nonnegative integers such that
\begin{enumerate}
\item $x_1=0$, and
\item for $1<i \leq n$, $x_i \leq \mathrm{asc}(x_1\dotsm x_{i-1})+1$.
\end{enumerate}
For example, 01234, 0120102, and 01013 are all ascent sequences, while 01024 is not since $\mathrm{asc}(0102)=2$.  Ascent sequences have been an increasingly frequent topic of study since Bousquet-M\'{e}lou, Claesson, Dukes, and Kitaev related them to $(2+2)$-free posets and enumerated the total number of ascent sequences \cite{BDCK10}.  Since then, various authors have connected ascent sequences to a number of other combinatorial objects \cite{CL11,DRKS11,DP10,KR11}; also see \cite[Section 3.2.2]{K11} for additional references.  The number of ascent sequences of length $n$ is given by the Fishburn numbers, Online Encyclopedia of Integer Sequences (OEIS) sequence A022493.

Given a string of integers $x=x_1\dotsm x_n$, the \emph{reduction} of $x$, denoted $\red(x)$ is the string obtained by replacing the $i$th smallest digits of $x$ with $i-1$.  For example, $\mathrm{red}(273772) = 021220$.  A \emph{pattern} is merely a reduced string.  We say that $x=x_1 \dotsm x_n$ \emph{contains} $p=p_1\dotsm p_k$ if there exists a subsequence of $x$ order-isomorphic to $p$, i.e., there exist indices $1 \leq i_1 < i_2 < \dotsm < i_k \leq n$ such that $\red(x_{i_1}x_{i_2}\dotsm x_{i_k}) = p$.  This is analogous to the classical definition of patterns for permutations, but here patterns may contain repeated digits, and patterns are normalized so that their smallest digit is 0 rather than 1. We write $\mathcal{A}(n)$ for the set of ascent sequences of length $n$ and $\mathcal{A}_B(n)$ for the set of ascent sequences of length $n$ avoiding all patterns in list $B$.  Also, we let $\mathrm{a}_B(n) = \left|\mathcal{A}_B(n)\right|$.

Pattern avoidance in ascent sequence was first studied by Duncan and Steingr\'{i}msson \cite{DS11}.  Their main results are given in Table \ref{DStable}.  Duncan and Steingr\'{i}msson also conjectured relationships between sequences avoiding 201, 210, 0123, 0021, or 1012 and other entries in the Online Encyclopedia of Integer Sequences \cite{OEIS}.  They provide the number of sequences avoiding 000, 100, 110, 120, or 201 of length at most 14, but without closed formula or conjecturally related objects.

\begin{table}
\begin{center}
\begin{tabular}{|c|c|c|c|}
\hline
Pattern $\sigma$&$\{\mathrm{a}_\sigma(n)\}_{n \geq 1}$&OEIS&Formula\\
\hline
001&\multirow{4}{*}{$1,2,4,8,16,32,\dots$}&\multirow{4}{*}{A000079}&\multirow{4}{*}{$2^{n-1}$}\\
010&&&\\
011&&&\\
012&&&\\
\hline
102&\multirow{3}{*}{$1,2,5,14,41,122,\dots$}&\multirow{3}{*}{A007051}&\multirow{3}{*}{$(3^{n-1}+1)/2$}\\
0102&&&\\
0112&&&\\
\hline
101&\multirow{3}{*}{$1,2,5,14,42,132,\dots$}&\multirow{3}{*}{A000108}&\multirow{3}{*}{$\frac{1}{n+1}\binom{2n}{n}$}\\
021&&&\\
0101&&&\\
\hline
\end{tabular}
\end{center}
\caption{Initial results for pattern-avoiding ascent sequences}
\label{DStable}
\end{table}

Mansour and Shattuck \cite{MS14} later computed the number of sequences avoiding 1012 or 0123 and showed that certain statistics on 0012-avoiding ascent sequences are equidistributed with other statistics on the set of 132-avoiding permutations.  They also considered 210-avoiding ascent sequences, and while they did not enumerate such sequences, they determined additional combinatorial structure.  Chen, Dai, Dokos, Dwyer, and Sagan \cite{CDDDS13} proved Duncan's and Steingr\'{i}msson's conjecture that the ascent statistic and the right-to-left minima statistic are each equidistibuted over 021-avoiding ascent sequences and 132-avoiding permutations.  A recent result of Callan \cite{C14} provides a bijection between 021-avoiding ascent sequences and Dyck paths which preserves several interesting statistics.  Callan, Mansour, and Shattuck also identified the complete equivalence class of pairs of length-4 patterns such that $\mathrm{a}_{\sigma,\tau}(n)$ is given by the Catalan numbers in \cite{CMS14}.

We will commonly make use of the following lemma of Duncan and Steingr\'{i}msson \cite{DS11}.  A word $x_1 \cdots x_n$ of nonnegative integers is called a \emph{restricted growth function}, or \emph{RGF}, if the first appearance of $k$ must be proceeded by a $k-1$ for each $k\geq 1$ (and it is quickly seen $k$ must be preceded by each of $0,1,2,\dotsc, k-1$).

\begin{lemma}[{\cite[Lemma 2.4]{DS11}}]\label{RGFlemma}
 Let $p$ be a pattern.  Then $\mathcal{A}_{p}(n)$ consists solely of RGFs if and only if $p$ is a subpattern of $01012$.
\end{lemma}

In this paper, we extend the work on pattern avoiding-ascent sequences by considering ascent sequences that avoid a pair of patterns of length 3 simultaneously.   There are 13 patterns of length 3 (000, 001, 010, 100, 011, 101, 110, and the six permutations), and thus there are 78 pairs of patterns.  Computation shows that these pairs produce at least 35 different avoidance sequences, 19 of which are new to OEIS.  Table \ref{BPtable} highlights a representative pattern set for each of the 16 sequences already found (as of August 2013) in OEIS for other reasons. We will use the following notation for common sequences:  $F_n$ is the $n$th Fibonacci number (with the convention $F_0=0$, $F_1=1$), $M_n$ is the $n$th Motzkin number, and $C_n$ is the $n$th Catalan number.  In the Table \ref{BPtable} we connect the appropriate ascent sequences to the formula or combinatorial object referenced in OEIS.

\begin{table}
\begin{center}
\begin{tabular}{|c|c|c|c|c|}
\hline
Patterns $B$&$\{\mathrm{a}_B(n)\}_{n \geq 1}$&OEIS&Formula&Result\\
\hline
010,021&$1,2,4,8,16,32,64,\dotsc$&A000079&$2^{n-1}$&Prop \ref{twos}\\
\hline
101,201&$1,2,5,14,42,132,429,\dotsc$&A000108&$C_n$&Prop \ref{catalan}\\
\hline
101,210&$1,2,5,14,41,122,365,\dotsc$&A007051&$(3^{n-1}+1)/2$&Prop \ref{threes}\\
\hline
000,012&$1,2,3,3,0,0,0,\dotsc$&trivial&&Prop \ref{allzeros}\\
\hline
000,011&$1,2,3,4,5,6,7,\dotsc$&A000027&$n$&Prop \ref{naturalnums}\\
\hline
000,001&$1,2,3,5,8,13,21,\dots$&A000045&$F_{n+1}$& Prop \ref{fib}\\
\hline
011,100&$1,2,4,7,11,16,22,\dots$&A000124&$n(n-1)/2+1$&Prop \ref{caterer}\\
\hline
001,100&$1,2,4,7,12,20,33,\dots$&A000071&$F_{n+2}-1$&Prop \ref{fib1}\\
\hline
001,210&$1,2,4,8,15,26,42,\dots$&A000125&$\binom{n}{3}+n$&Prop \ref{cake}\\
\hline
000,101&$1,2,4,9,21,51,127,\dots$&A001006&$M_n$&Prop \ref{motzkin}\\
\hline
100,101&$1,2,5,13,35,97,275,\dots$&A025242&&Prop \ref{GenCatalan}\\
\hline
021,102&$1,2,5,13,32,74,163,\dots$&A116702&$3\cdot2^{n-1}-\binom{n+1}{2}-1$&Prop \ref{perm1}\\
\hline
102,120&$1,2,5,13,33,81,193,\dots$&A005183&$(n-1)2^{n-2}+1$&Prop \ref{perm2}\\
\hline
101,120&$1,2,5,13,33,82,202,\dots$&A116703&&Prop \ref{perm3}\\
\hline
101,110&$1,2,5,13,89,233,610,\dots$&A001519&$F_{2n-1}$&Prop \ref{bifib}\\
\hline
201,210&$1,2,5,15,51,188,731,\dots$&A007317&$\sum\limits_{k=0}^{n-1}\binom{n-1}{k}C_k$&Prop \ref{binomcat}\\
\hline
\end{tabular}
\end{center}
\caption{Ascent sequences avoiding a pair of patterns of length 3}
\label{BPtable}
\end{table}

\section{Proofs by elementary methods}\label{patternset}

First we make an observation: in the class of pattern-avoiding permutations, avoiding one pattern of length 3 always produces a different sequence than avoiding two patterns of length 3.  (In particular, the number of length 3 permutations avoiding a single pattern is 5, while the number of length 3 permutations avoiding a pair of patterns of length 3 is 4.)  With ascent sequences this is no longer the case, as we see in the following two propositions.

\begin{prop}\label{twos}
$\mathrm{a}_{010,021}(n)=2^{n-1}$
\end{prop}

\begin{proof}
Duncan and Steingr\'{i}msson prove that $\mathcal{A}_{010}(n)=\mathcal{A}_{10}(n)$, since both restrictions force the ascent sequences to be weakly increasing.  Clearly $\mathcal{A}_{p,q}(n)=\mathcal{A}_{p}(n)$ for any $n$ if $p$ is contained in $q$, and since $10$ is contained in $021$ the result follows from  Duncan and Steingr\'{i}msson's result that $\left|\mathcal{A}_{010}(n)\right|=2^{n-1}$.
\end{proof}

As remarked in the proof, we have the following generalization:

\begin{cor}
If $p$ is a pattern containing 10, then  $\mathcal{A}_{010,p}(n)~=~\mathcal{A}_{010}(n)$.
\end{cor}

Pattern-avoiding ascent sequences exhibit a broader phenomenon of \emph{superfluous restrictions}, that is, patterns $p$ and $q$ such that $\mathcal{A}_{p,q}(n) = \mathcal{A}_{p}(n)$ even when $p$ is not contained in $q$.  We now consider a few more examples of such superfluous restrictions.  

\begin{prop}\label{catalan}
$\mathrm{a}_{101,201}(n)=C_n$
\end{prop}

\begin{proof}
We claim that if $x$ avoids 101 then $x$ also avoids 201.  By Lemma 2.4 of \cite{DS11}, the 101-avoiding ascent sequences are restricted growth functions (or RGFs), meaning that that the first appearance of $k$ must be proceeded by each of $0, 1, 2, \dotsc, k-2$, and $k-1$.  Suppose $x$ contains a 201 pattern with letters $cab$ for $a<b<c$.  The RGF restriction implies the $c$ must be preceded by a $b$, and so we see $x~=~w_1~b~w_2~c~w_3~a~w_4~b~w_5$ for intervening words $w_i$ and so contains 101 by the $bab$.  Thus we see that the 201 is a superfluous restriction when paired with 101, and so we may apply Duncan and Steingr\'{i}msson's result that $\left|\mathcal{A}_{101}(n)\right|=C_n$. 
\end{proof}

Similar reasoning yields following generalization:

\begin{cor}
We have $\mathcal{A}_{101, q}(n) = \mathcal{A}_{101}(n)$ for any pattern $q$ containing $201$.
\end{cor}

We next move to another unusual phenomenon where $|\mathcal{A}_{p,q}(n)| = |\mathcal{A}_{r}(n)|$ for three patterns $p$, $q$, and $r$ all of the same length.  Theorem 2.8 of \cite{DS11} states that $\mathrm{a}_{102}=(3^{n-1}+1)/2$.  The same enumeration sequence appears for ascent sequences avoiding a pair of patterns in a less obvious way than superfluous restrictions, as the following theorem shows.

\begin{prop}\label{threes}
$\mathrm{a}_{101,210}(n) = \frac{3^{n-1}+1}{2}$
\end{prop}

\begin{proof}
Lemma 2.7 of \cite{DS11} proves that there are $\frac{3^{n-1}+1}{2}$ ternary strings of length $n-1$ with an even number of 2s, and Thereom 2.8 puts $\mathcal{A}_{102}(n)$ in bijection with these strings.  We now put members of $\mathcal{A}_{101,210}(n)$ in bijection with the same ternary strings.

As mentioned in the proof of Proposition \ref{catalan}, an ascent sequences avoiding 101 must be an RGF.  Therefore, we are considering a subset of the RGFs with additional restrictions:
 \begin{enumerate}[(i)]
  \item   If $x$ avoids 101, and $x_i<x_{i+1}$ then $x_{j}< x_{i+1}$  for any $j<i$ (since otherwise a 101 pattern appears).  In other words, the second letter of an ascent must be the first time a letter with that value appears in $x$.  

 \item If $x$ is an RGF which avoids 101, then whenever $x_i > x_{i+1}$ there must be some $j<i$ such that $x_j = x_{i+1}$.  Let $j^*$ be the largest index less than $i$ such that $x_{j^*} = x_{i+1}$.  Then $x_{j^*} < x_{j^*+1}$, since otherwise $x_{j^* } > x_{j^*+1}$ implies $\red(x_{j^*} x_{j^*+1} x_{i+1}) = 101$ and maximality of $j^*$ implies $x_{j^*}\neq x_{j^*+1}$.  If $x$ also avoids 210, then we can further state $x_{j^*+1} <x_i$ or else a 210 is formed by $x_{j^*+1}x_i x_{i+1}$.  

 \item If $x$ avoids both 101 and 210 then we can comment on pairs of descents.  Suppose $x_i > x_{i+1}$ and $x_j > x_{j+1}$ for $i<j$.  Then $x_{i+1}\leq x_{j+1}$ or else $x_i x_{i+1} x_{j+1}$ forms a 210.  Furthermore if $x_{i^*}$ is the rightmost letter to the left of $x_i$ which equals $x_{i+1}$, and $x_{j^*}$ is the rightmost letter to the left of $x_{j+1}$ which equals $x_{j+1}$, then $i^* < j^*$.  If $x_{i+1} = x_{j+1}$ then $j^*\geq i+1$, and if $x_{i+1} < x_{j+1}$ then having $j^*< i^*$ would imply a 101 pattern by $x_{j^*} x_{i^*} x_{j+1}$.  
 \end{enumerate}

Observations (i) and (ii) imply that for any pair of adjacent letters $x_i x_{i+1}$ for $x\in \mathcal{A}_{101,210}(n)$ there are four possible behaviors:
 \begin{enumerate}[(A)]
  \item $x_i = x_{i+1}$
  \item $x_i > x_{i+1}$
  \item $x_i < x_{i+1}$ where $x_i=d$ is the last occurrence of the value $d$ in $x_1\dotsm x_{i}$ before there is a descent of the form $x_j d$. In terms of observation (ii) above, $x_i$ plays the role of $x_{j^*}$ for some later descent.
  \item $x_i < x_{i+1}$ where $x_i=d$ is \emph{not} the last occurrence of the value $d$  in $x_1\dotsm x_{i}$ before there is a descent of the form $x_j d$
 \end{enumerate}
 These behaviors imply an encoding of any ascent sequence by a word in $w\in \{A,B,C,D\}^{n-1}$, where $w_i$ takes the value corresponding to the behavior of $a_i$ as listed above.  Observation (ii) above implies that each descent is preceded by an ascent and so behaviors (B) and (C) appear equally often with the first instance of (C) preceding the first (B).  The second half of observation (iii) above implies that the instances of (C) and (B) will in fact alternate, with each (C) followed by a (B) before another (C) occurs.  Therefore any $x\in \mathcal{A}_{101,210}(n)$ is encoded as a word in $\{A,B,C,D\}^{n-1}$ where $C$ and $B$ alternate with $C$ appearing first.  We will call such words \emph{$CB$-alternating}.  

There is an obvious bijection between $CB$-alternating words and ternary strings of the same length with an even number of 2's under the mapping $A\mapsto 0$, $B\mapsto 2$, $C\mapsto 2$, and $D\mapsto 1$.  Therefore it remains to show that the encoding of ascent sequences by $CB$-alternating words is bijective when the domain is restricted to $\mathcal{A}_{101,210}(n)$, which we will do by constructing the inverse. 

Let $w_1 \dotsc w_{n-1}$ be a $CB$-alternating word, and we will construct the corresponding $x\in\mathcal{A}_{101,210}(n)$.  Let $x_1=0$, and for $2\leq i \leq n$ determine $x_i$ as follows.  If $w_{i-1}=A$, let $x_i=x_{i-1}$.  If $w_{i-1}=C$ or $w_{i-1}=D$, let $x_i = \max(x_1, \dotsc, x_{i-1})+1$.  If $x_{i-1}=B$, let $x_i$ take the same value as $x_{j}$ where $w_j=C$ for the largest value of $j<i$.  
\end{proof}

To  illustrate the map, the ascent sequence $012131114 \in \mathcal{A}_{101,210}(9)$ corresponds to the word $DCBCBAAD$.   Inversely, $CAABDDCDB$ corresponds to the ascent sequence 0111023453.

\bigskip

We now consider a set of restrictions that ultimately leads to the all zeros sequence, producing an analogue of the Erd\H{o}s-Szekeres theorem.  We consider an all-zeros pattern $0^a = \underbrace{00\dotsc0}_{a}$ and the strictly increasing pattern $012\dotsm b$.  We begin with the pair of patterns of length 3.

\begin{prop}\label{allzeros}
$\mathrm{a}_{000,012}(n)=\begin{cases}
1&n=1\\
2&n=2\\
3&n=3\text{ or }n=4\\
0&n\geq 5\\
\end{cases}$
\end{prop}

\begin{proof}
All ascent sequences of length less than 3 avoid these two patterns, which gives the first two cases.  It is easily checked that  $\mathcal{A}_{000,012}(3)=\{001, 010, 011\}$ and $\mathcal{A}_{000,012}(4)=\{0010, 0101, 0110\}$, which gives the third case.

Now, consider an arbitrary ascent sequence of length greater than 4.  It is easily seen that any ascent sequence avoiding 012 cannot have any digit greater than or equal to 2, since the first ascent must be a literal 01.   Any sequence with at least 5 digits consisting of only 0s and 1s either has three 0s or three 1s, which creates a 000 pattern in either case.  Therefore no ascent sequence of length 5 or greater can avoid both 000 and 012.
\end{proof}

A similar argument shows $\mathrm{a}_{0^a, 012}(n) = 0$ for $n \geq 2a-1$.  Avoiding $012$ restricts the ascent sequence to using only the digits $0$ and $1$, and avoiding $0^a$ means that we can only use each of these at most $a-1$ times.  Hence we can have at most $(a-1)(2) = 2a-2$ letters in any ascent sequence avoiding $0^a$ and $012$.  Thus we have proven the following corollary.

\begin{cor}
 $\mathcal{A}_{0^a, 012}(n)=\emptyset$ for $n\geq 2a-1$.
\end{cor}

We can generalize further to see $\mathcal{A}_{0^a,012\dotsm b}(n)$ is also eventually empty.  If $x$ avoids $0^a$ and $012\dotsm b$ then the largest letter possible in $x$ is $(a-1)(b-2)+1$:  this largest letter could be preceded by at most $b-1$ different smaller values, and using at most $a-1$ copies each of these $b-1$ smaller values could only generate at most $(a-1)(b-2)$ ascents (as witnessed by repeating the increasing pattern $a-1$ times).  If every value in $x$ appears $a-1$ times, and every value from $0$ to the maximum $(a-1)(b-2)+1$ appears, then there can be at most $(a-1)((a-1)(b-2)+2)$ letters in $x$.  Thus we arrive at the following corollary:

\begin{cor}
  $\mathcal{A}_{0^a, 012\dotsm b}(n)=\emptyset$ for $n\geq (a-1)((a-1)(b-2)+2)+1$.
\end{cor}

Note that the above bound is tight.  An ascent sequence (of maximal length) avoiding $0^a$ and $012\dotsm b$ is witnessed by:
$$ \bigl(012\dotsm (b-2)\bigr)^{a-1} \, \bigl( (a-1)\cdot(b-2)+1 \bigr)^{a-1} \bigl( (a-1)\cdot(b-2) \bigr)^{a-1} \dotsm \bigl(b-1\bigr)^{a-1},$$
where we use the shorthand of $(w)^a$ to mean $a$ copies of the word $w$.

\bigskip

In the rest of this section we give constructive or bijective explanation to why various combinatorial sequences arise when considering particular pattern-avoiding ascent sequences.

\begin{prop}\label{naturalnums}
$\mathrm{a}_{000,011}(n) = n$ for $n \geq 1$.
\end{prop}

\begin{proof}
Given $a \in \mathcal{A}_{000,011}(n)$, one may always append $\mathrm{asc}(a)+1$ onto the end of $a$.  Also, the monotone increasing sequence is always a member of $\mathcal{A}_{000,011}(n)$ and one may also append 0 onto the end of it.  These are precisely the $n+1$ members of $\mathcal{A}_{000,011}(n+1)$.
\end{proof}



\begin{prop}\label{fib}
$\mathrm{a}_{000,001}(n) = F_{n+1}$
\end{prop}


This result also appears in \cite{MS15} as Proposition 2.1.  We include a proof below for completeness.  Mansour and Shattuck also show in Proposition 3.3 of \cite{MS15} that $\mathrm{a}_{000,010} = F_{n+1}$, which follows from the fact that $\mathrm{a}_{000,10}=F_{n+1}$.

\begin{proof}
Consider $a \in \mathcal{A}_{000,001}(n)$.  Because $a$ avoids 000, there are either 1 or 2 copies of each digit that appears in $a$.  Because $a$ avoids 001, if there are two copies of $k$ and $k$ is not the largest digit, then one digit occurs before the largest digit, and one appears after.  In fact, such sequences have the form $a=012\dotsm \ell(a_d)$, where $\ell$ is the largest digit, and $a_d$ is all remaining digits in decreasing order.

By listing, we see that there is one such sequence of length $1$, and there are two such sequences of length $2$.  The number of sequences one copy of the largest digit is given by $\mathrm{a}_{000,001}(n-1)$ and the number of sequences with two copies of the largest digit is given by $\mathrm{a}_{000,001}(n-2)$.
\end{proof}

\begin{prop}\label{caterer}
$\mathrm{a}_{011,100}(n)=\binom{n}{2}+1$
\end{prop}

\begin{proof}
Consider $a \in \mathcal{A}_{011,100}(n)$.  Because $a$ avoids $011$, any nonzero letter of $a$ appears exactly once.  It immediately follows that the nonzero letters of $a$ form a strictly increasing subsequence.  Furthermore since $a$ avoids $100$, there can be at most one $0$ following the first nonzero letter of $a$.  Therefore $a$ has one of the two following structures:

\begin{tabular}{cl}
  Case 1: & $\underbrace{00\dotsc0}_{\ell} \, 123\dotsc (n-\ell)$ \\
  Case 2: &   $\underbrace{00\dotsc0}_{\ell} \, 123\dotsc m\,0\,(m+1)(m+2)\dotsc (n-\ell-1)$
\end{tabular}

There are clearly $n$ sequences with the structure in Case 1, including the all-zeroes sequence.  For Case 2, $\ell$ and $m$ range over the parameter space $\ell \geq 1$, $m\geq 1$, and $\ell + m + 1 < n$.  Note that this includes the sequences of the form $0\dotsc0\,123\dotsc (n-\ell-1)\,0$.  It follows that there are $1+2+\dotsm+(n-2) = \binom{n-1}{2}$ such ascent sequences.  Therefore there are $n+\binom{n-1}{2} = \binom{n}{2}+1$ ascent sequences in $\mathcal{A}_{011,100}(n)$.

Combinatorially the formula $\binom{n}{2}+1$ can be described as follows.  For the $n$ slots for the letters of the sequence, we mark two of them.  If the first slot is marked, then we write zeros until the second marked slot and then proceed with $12\dotsc$ until we reach the end.  This results in a sequence as described in Case 1.  If the first slot is unmarked, then we write zeros from the first slot until the first marked slot, then switch to writing the strictly increasing sequence $123\dotsc$ with the exception of the second marked slot which gets a $0$.  This method fails to create the all zeros sequence, which is the extra ``$+1$'' of the formula.
\end{proof}

\begin{prop}\label{fib1}
$\mathrm{a}_{001,100}(n)=F_{n+2}-1$
\end{prop}

\begin{proof}
Consider $a \in \mathcal{A}_{001,100}(n)$ and suppose $i$ is the largest digit in $a$.  For any digit $d$ other than $i$, there are at most 2 copies of $d$: otherwise, there will either be two copies of $d$ before the first $i$, or there will be two copies of $d$ after the first $i$. Since there can be at most one copy of each other digit $d$ before the first $i$, $a$ must have the form $01\dotsm i a_d$, where $a_d$ is a decreasing arrangement of the remaining digits.

Now, if there is one copy of largest digit, delete it to get a member of $\mathcal{A}_{001,100}(n-1)$ with one or two copies of new largest digit.
If there are two copies of largest digit, delete both to get a member of $\mathcal{A}_{001,100}(n-2)$ with one or two copies of new largest digit.
If there are three to $n-1$ copies of largest digit, delete all $0s$ and reduce to get a member of $\mathcal{A}_{001,100}(n-1) \cup \mathcal{A}_{001,100}(n-2)$ with 3 or more copies of the new largest digit.
We have now related all but one member (the all zeros sequence) of $\mathcal{A}_{001,100}(n)$ to a unique member of $\mathcal{A}_{001,100}(n-1) \cup \mathcal{A}_{001,100}(n-2)$, which shows that $\mathrm{a}_{001,100}(n) = \mathrm{A}_{001,100}(n-1)+\mathrm{A}_{001,100}(n-2)+1$, the same recurrence satisfied by $F_{n+2}-1$.  Upon verifying the base cases the proof is complete.
\end{proof}

\begin{prop}\label{cake} 
$\mathrm{a}_{001,210}(n)=\binom{n}{3}+n$
\end{prop}

\begin{proof}
Since 001 is contained in 01012, Lemma 2.4 of \cite{DS11} implies that all ascent sequences in $\mathcal{A}_{001,210}(n)$ are RGFs.  Consider $a \in \mathcal{A}_{001,210}(n)$.  Either $a$ has a descent, or it does not.

If $a$ does not have a descent, $a$ automatically avoids 210, so we need only consider what an increasing ascent sequence that avoids 001 looks like.   In this case, as soon as $a$ has a repeated digit $d$, all remaining digits must be equal to $d$.  Further, $a$ must have the form $0123\dotsm dd\dotsm d$ since $a$ is an RGF.  There are $n$ such sequences.

Now, suppose that $a$ does have a descent.  Let $d_1 d_2$ be the first descent in $a$.  Consider digit $a_i$ appearing after $d_2$.  Then $a_i\geq d_2$, otherwise $d_1 d_2 a_i$ form a 210 pattern.  Further, $a_i \leq d_2$, for a slightly more subtle reason:  Because $a$ is an RGF, a copy of digit $d_2$ also appears before $d_1$.  If $a_i > d_2$, then $d_2 d_2 a_i$ form a 001 pattern.  Therefore, $a_i=d_2$.  Therefore if $a$ has a descent $d_1 d_2$, then $a$ has the form 
$$012\dotsm (d_1 -1) (d_1)^k(d_2)^{(n-k-d_1)}.$$
Thus, $a$ is uniquely defined by choosing $k$, $d_1$, and $d_2$.  Clearly $0 \leq d_2<d_1<n$.    So, choose $d_2$, as one of the $n-1$ values between $0$ and $n-2$.  Once $d_2$ is chosen, since $d_2<d_1<n$, there are $n-d_2-1$ choices for $d_1$.  Finally, $n-k-d_1\geq 1$ and $k \geq 1$, so $n-d_1\geq k+1 \geq 2$ and $n-d_1-1 \geq k \geq 1$, which indicates that once $d_1$ is chosen, there are $n-d_1-1$ choices for $k$.  Equivalently, let $i=d_2+1$, $j=d_1+1$. and $m=n+1-k$.  Now $\{i,j,m\}$ is a uniquely defined set of 3 distinct elements in $\{1,2,\dots,n\}$, which ensures there are precisely $\binom{n}{3}$ members of $\mathcal{A}_{001,210}(n)$ that have a descent.
\end{proof}

\begin{prop}\label{motzkin}
$\mathrm{a}_{000,101}(n)=M_n$, where $M_n$ denotes the $n^{th}$ Motzkin number.
\end{prop}

\begin{proof}
Let $a \in \mathcal{A}_{000,101}(n)$ and consider the final digit $d=a_n$.  If $a$ has only one copy of $d$, delete it to obtain a sequence in $\mathcal{A}_{000,101}(n-1)$.  If there are two copies of $d$, everything between these two copies must be larger than $d$ (lest $a$ have a 101 pattern).  After reducing, these digits correspond to an ascent sequence in $\mathcal{A}_{000,101}(n-2-i)$ for some $i\geq 0$, and the subsequence before the first occurrence of $d$ is a member of $\mathcal{A}_{000,101}(i)$.

We have just shown that  for $f(n):=\mathrm{a}_{000,101}(n)$
$$f(n) =f(n-1) + \sum\limits_{i=0}^{n-2} f(i)\cdot f(n-2-i),$$
 which is the recurrence for the Motzkin numbers.  After verifying initial conditions the proof is complete.
\end{proof}

\section{A bijection with Dyck words}

Recall that a Dyck word of semilength $n$ is a word in $\{U,D\}^{2n}$ with $n$ $U$s and $n$ $D$s such that the first $k$ letters have at least as many $U$s as $D$s for all $1\leq k \leq 2n$.  A Dyck word $w=w_1 w_2 \dotsm w_{2n}$ is said to contain $DDUU$ if there are (consecutive) letters $w_i w_{i+1} w_{i+2} w_{i+3} = DDUU$, otherwise $w$ is said to avoid $DDUU$.  These have a natural interpretation as lattice paths with steps $U=(1,1)$ and $D=(1,-1)$ and the path starts and ends on, but never dips below, the $x$-axis.  Thus we have cause to refer to a $U$ as an ``upstep'' and a $D$ as a ``downstep.''  

\begin{prop}\label{GenCatalan}
$\sum_{n \geq 1} \mathrm{a}_{100,101}(n)x^n = \frac{(1-x)^2-\sqrt{1-4x+2x^2+x^4}}{2x^2}$.  Equivalently, the set $\mathcal{A}_{100,101}(n)$ is in bijection with the set of DDUU-avoiding Dyck words of semilength $n$.
\end{prop}

Note that this corresponds to OEIS sequence A025242.   This particular sequence, dubbed ``generalized Catalan numbers,'' is explored by Mansour and Shattuck in \cite{MS11}.   Mansour and Shattuck considered the number of RGFs (stemming from set-partitions),  which avoid certain pairs of patterns of length 4.  In particular they listed all pairs of patterns of length 4 which result in this enumeration sequence.   One such pair is $(1211, 1212)$, which in our convention would be written (0100, 0101).  Ignoring the initial 0's shows a clear connection to our pair $(100,101)$.  Indeed, by Lemma \ref{RGFlemma} the ascent sequences avoiding 101 must be RGFs.  Furthermore, because the ascent sequences are RGFs, one contains 100 if and only if it contains 0100 and likewise one contains 101 if and only if it contains 0101.  Therefore the set of ascent sequences avoiding (100, 101) equals the set of ascent sequences avoiding (0100, 0101), which is also the set of RGFs avoiding (0100, 0101).  

The work of Mansour and Shattuck yields the generating function given above, and the work of Sapounakis et al. in \cite{STT07} connects that same generating function to the Dyck words of semilength $n$ with no DDUU factor.  

The remainder of this section is given over to a \emph{bijective} proof of the proposition.

Let $\mathcal{D}(n)$ be the set of Dyck words of semilength $n$ with no DDUU factor.  We will construct a bijection from $\mathcal{D}(n)$ to $\mathcal{A}_{100,101}(n)$ via the heights of the upsteps in conjunction with a ``lifting'' procedure designed to eliminate copies of 100 patterns.

In a Dyck word $d_1 \dotsm d_{2n}$, the \emph{height} of an upstep at $d_i$ is the number of $U$s minus the number of $D$s in $d_1\dotsm d_{i-1}$.   In the lattice path interpretation, this is the starting height of the upstep.  In the example below, each $U$ is marked with its height.

\begin{tabular}{*{14}{c}}
U&U&U&D&D&U&D&U&U&D&D&D&U&D \\
0&1&2& & &1& &1&2&&&&0& \\
\end{tabular}

Let $w(d) = w_1 \dotsm w_n$ where $w_i$ is the height of the $i^{th}$ upstep in Dyck word $d$.  This U-height word $w(d)$ has several important properties for our purposes:
 \begin{enumerate}
 \item[(1)] $w_i < w_{i+1}$ implies $w_{i+1} = w_{i} + 1$, and therefore $w$ is an RGF starting from 0.
 \item[(2)] Any Dyck path $d$ is uniquely determined by its U-height word, (i.e., $w(d) = w(d')$ implies $d=d'$).
 \item[(3)] If $d$ has no DDUU factor, then $w_i < w_{i+1}$ implies $w_{i-1} = w_i$ or $w_{i-1} = w_{i} - 1$ for $i\geq 2$.
 \end{enumerate}
Note that properties (1) and (2) hold for all Dyck paths, not just those which avoid DDUU.

Property (1) follows immediately from the observation that an upstep increases the height of the next upstep by at most 1.  Also note that $w_1=0$ since any Dyck word begins with an upstep.  For any digit $k>0$, its first appearance in $w$ must be immediately preceded by $k-1$ (a somewhat stronger condition than the RGF condition).

To prove Property (2), we will reconstruct Dyck word $d=d_1 \dotsm d_{2n}$ so that $w(d)$ matches a given $w_1 \dotsm w_n$.  As for any Dyck word, $d_1=U$ and we proceed to complete $d$ from left to right.  If $w_i \geq w_{i+1}$ then the $i^{th}$ upstep is followed by $w_i - w_{i+1} + 1$ downsteps, followed by the $(i+1)^{th}$ upstep.  If $w_i < w_{i+1}$ then by property (1) we see that the $i^{th}$ and $(i+1)^{th}$ upsteps are consecutive.  Last there are $w_{n}+1$ downsteps at the end of $d$.  In this way we have inverted the function $w(d)$ and so $d$ is uniquely determined.

Property (3) follows from the proof of Property (2) as follows.  As noted before, $w_i < w_{i+1}$ implies that the $i^{th}$ and $(i+1)^{th}$ upsteps of the Dyck path $d$ are consecutive.  Property (2) implies that if $w_{i-1} < w_{i}$ then  $w_{i-1} = w_{i} - 1$.  If $w_{i-1}>w_{i}$ then the proof of Property (2) implies there are $w_{i-1} - w_{i} + 1$ consecutive downsteps immediately preceding the $i^{th}$ upstep.  In particular there would be at least 2 downsteps preceding the UU formed by the $i^{th}$ and $(i+1)^{th}$ upsteps, which creates a DDUU factor.  Property (3) follows.

We need one additional property to prove the bijection constructed below does what is needed.  It states that, in a very specific sense, the first copy of 100 must finish before the first copy of 101 finishes.
\begin{enumerate}
 \item[(4)] Let $w=w(d)$ for a DDUU-avoiding Dyck path $d$.  Choose index $j$ minimally such that there exist indices $i$ and $k$ such that $i<j<k$ and $w_i w_j w_k$ forms a copy of either 100 or 101.  Then the minimal such $k$ will create a 100 pattern.
\end{enumerate}

To illustrate property (4), consider the $w=012112$, which generates the DDUU-avoiding Dyck path $d=UUUDDDUDUUDDD$.  There are three copies of 100 and 101 in $w$, occurring at indices $(i,j,k)\in \{(3,4,5), (3,4,6), (3,5,6)\}$.  Thus the minimal such $j$ is 4, and among those copies with $j=4$, the copy with minimal $k$ is $(3,4,5)$ which is a copy of 100 not 101.



Note that property (4) implies that if $w$ avoids 100 then $w$ must also avoid 101.  This motivates the focus on removing 100 patterns in the bijection.

The proof of Property (4) is as follows.  Choose $i$, $j$, and $k$ as described.  The $j^{th}$ upstep of $d$ is followed either by an upstep or downstep, which we consider in cases below.

\emph{Case 1:} Assume the $j^{th}$ upstep is followed by a downstep.  Then $w_j \geq w_{j+1}$.  If $w_j = w_{j+1}$, then $w_i w_j w_{j+1}$ is a copy of 100.  If $w_j > w_{j+1}$, then Property (1) implies the upsteps following the $(j+1)^{th}$ upstep will reach height $w_j$ (thus forming a 100 pattern) before reaching the greater height $w_i$ (which would create a 101) pattern.  Therefore in this case the result is proven.

\emph{Case 2:} Assume the $j^{th}$ upstep is followed by another upstep, which we will show results in a contradiction with how $j$ is chosen.  Then $w_{j}+1=w_{j+1}$.  
The UU formed by these two upsteps is part of a larger block of consecutive upsteps starting at some index $c$.  That is, $c$ is the minimal index such that $w_c < w_{c+1} < \dotsm < w_{j} < w_{j+1}$.  Note $c>i$ since $w_i > w_j$ necessitates some intervening downsteps.  By property (3), the DDUU-avoidance implies $w_{c-1} = w_c$.  Therefore $w_{i} w_{c-1} w_{c}$ is a copy of 100 for $c-1 < j$, contradicting the minimality of $j$.  

This concludes the proof of Property (4).

Now, we define our map $\phi : \mathcal{D}(n) \to \mathcal{A}_{100,101}(n)$ as follows.  It works by removing copies of 100 from left to right by ``lifting'' subwords of $w(d)$.  In the process we generate words which we denote $w^0, w^1, w^2, \dotsc$, the last of which will be our $\phi(d)$.

\begin{enumerate}
\item Given $d \in \mathcal{D}(n)$, construct $w(d)$.
\item If $w(d)$ contains no forbidden 100 or 101 patterns, then $\phi(d)=w(d)$.
\item If $w(d)$ contains a forbidden pattern, then Property (4) implies there is a 100 pattern and we proceed as follows.
\begin{enumerate}
\item Let $w^{0}=w(d)$ and let $m=0$.
\item Let $i<j<k$ be the indices so that $w_i w_j w_k$ is a copy of 100 in $w^{m}$, chosen such that $j$ is minimal and $k$ is minimal as in Property (4).  
\item Let $\ell$ be the smallest index such that $k<\ell$ and $w^{m}_k> w^{m}_\ell$.
\item Let $w^{m+1}$ be the word formed by adding $\max(w^{m}_1 \dotsm w^{m}_{k-1})+1-w^m_{k}$ to each digit of $w^{m}_{k} w^{m}_{k+1} \dotsm w^{m}_{\ell-1}$.  Also let $w^{m+1}_{i} = w^{m}_{i}$ for $i<k$ and $\ell \leq i \leq n$.
\item If $w^{m+1}$ has no forbidden pattern, then $\phi(d) = w^{m+1}$.  If not, increase $m$ by 1 and return to step (b).
\end{enumerate}
\end{enumerate}

We will call the action in step (d) a ``lift'' and the subword $w^{m}_{k} w^{m}_{k+1} \dotsm w^{m}_{\ell-1}$ a ``lifted part.''  Each lift removes at least one copy of 100, and cannot create any new copies.  Since $w^{0}$ has finitely many copies of 100 this process must terminate and so $\phi$ is well-defined.   

Consider the following example of $\phi$ in action.  Let $$d=UUUDDUDUUDDUDUUDDDUDUUDD,$$ and so $w(d) = 012112112001$.  There is a copy of 100 (several, in fact), and so we must perform the following lifts.  At each stage the first copy of 100 is in boldface and the portion to be lifted is underlined.

\begin{tabular}{l}
$w^{0} = 01\mathbf{21\underline{1}}\underline{2112}001$ \\
$w^{1} = 01213\mathbf{43\underline{3}}\underline{4}001$ \\
$w^{2} = 0\mathbf{1}2134356\mathbf{0\underline{0}}\underline{1}$ \\
$w^{3} = 012134356078$ \\
\end{tabular}

We next need to show $\phi(d) \in \mathcal{A}_{100,101}(n)$.  First observe $w^{m+1}$ an RGF so long as $w^{m}$   is an RGF since  $w^{m+1}_k = \max(w^{m+1}_1 \dotsm w^{m+1}_{k-1})+1$.  Therefore by Property (1), $\phi(d)$ will be an RGF and thus an ascent sequence.  By construction, $\phi(d)$ avoids 100 as well, and so can be sure $\phi(d) \in \mathcal{A}_{100}(n)$.

It remains to show $\phi(d)$ avoids 101.  Suppose for induction that $w^{m}$ satisfies Property (4), that is,  if $j$ is chosen minimally such that $w^{m}_i w^{m}_j w^{m}_k$ is a copy of 100 or 101, then the minimal such $k$ will create a copy of 100.  We will prove the same is true for $w^{m+1}$, where $w^{m}_k \dotsm w^{m}_{\ell-1}$ is the relevant lifted part.  Suppose  $w^{m+1}_a w^{m+1}_b w^{m+1}_c$ is a copy of 100 or 101 for minimal $b$.   Notice that the lifted part reduces to a word which itself is a U-height word for a DDUU-avoiding Dyck path of semilength $\ell - k$.  Therefore the Property (4) applies to this lifted part and so we see that first 101 pattern strictly within that lifted part (if there is any) would be preceded by a 100. Suppose, then that the lifted part avoids 100 (and thus 101) and thus the offending $w^{m+1}_a w^{m+1}_b w^{m+1}_c$ lies partly in the lifted part and partly outside.  Since the lifted part is greater than all digits to its left,  and the non-lifted digits remain the same, it follows that $k \leq a \leq \ell-1$ and therefore $c>\ell$.  Since $w^{m+1}_\ell < w^{m+1}_c$, every intermediate value must appear since $w^{m+1}_\ell w^{m+1}_{\ell} \dotsm w^{m+1}_c = w_\ell w_{\ell} \dotsm w_c$ which satisfies Property (1).   Therefore there must be a digit with the same value as $w^{m+1}_b$ appearing before $w^{m+1}_c$, which would create a 100 before a 101. 

Observe that  if $x=\phi(d)$ has letters $x_i$ and $x_{i+1}$ so that $x_{i}+1< x_{i+1}$, then a lift must have taken place that affected the $(i+1)^{th}$ letter but not the $i^{th}$.  This provides the key to the inverse mapping which acts as follows:

\begin{enumerate}
\item Given $x \in \mathcal{A}_{100,101}(n)$, let $\ell$ be the largest integer such that $x_{\ell-1}+1<x_\ell$ if such an integer exists.
\item If there is no such $\ell$, then $\phi^{-1}(x) = w^{-1}(x)$.
\item If there is such an $\ell$, then we proceed as follows
\begin{enumerate}
\item Let $x^0=x$ and set $m=0$.
\item Construct $x^{m+1}$ so that 

$x^{m+1}_i = \begin{cases}
x^m_i & i < \ell \\
x^m_i & i \geq \ell \text{ and } x^m_i < \left(x^m_{\ell}-x^m_{\ell-1} \right)\\
x^m_i - \left(x^m_{\ell}-x^m_{\ell-1} \right) & \text{otherwise}\\
\end{cases}$
\item If $x^{m+1}$ has no integer $\ell$ such that $x^{m+1}_{\ell-1}+1<x^{m+1}_\ell$, then $\phi^{-1}(x) = w^{-1}(x^{m+1})$.  If not, increment $m$ and return to step b.
\end{enumerate}
\end{enumerate}

This process terminates because there are a finite number of indices where  $x_{\ell-1}+1<x_\ell$, we remove such an occurrence with each iteration, and decrementing the tail of the word cannot produce any new indices with large jumps.  Furthermore, the end result is guaranteed to result in a word satisfying Property (1), and therefore corresponds to the $U$-height word for a Dyck path.

For example, consider $x=012134356078$.  Then we compute $\phi^{-1}(x)$ as follows, where at each stage the large jump is marked in bold face and the portion which is lowered is underlined.

\begin{tabular}{l}
$x^0 = 012134356\mathbf{0\underline{7}}\underline{8}$ \\
$x^1 = 012134\mathbf{3\underline{5}}\underline{6}001$ \\
$x^2 = 012\mathbf{1\underline{3}}\underline{4334}001$ \\
$x^3 = 012112112001$ \\
\end{tabular}

\noindent From $x^{3} = 012112112001$ we can reconstruct the original Dyck path to see $\phi^{-1}(x)=UUUDDUDUUDDUDUUDDDUDUUDD$.

Thus we have provided a bijection between DDUU-avoiding Dyck paths of semilength $n$ and ascent sequences avoiding 100 and 101 of length $n$.


\section{Counting by generating trees}

Generating trees have proven an extremely useful tool in pattern avoidance for permutations.  See \cite{V06} for a more detailed description and history.  In Propositions \ref{perm1}, \ref{perm2}, and \ref{perm3}  we employ generating trees to put pattern-avoiding ascent sequences in bijection with pattern-avoiding permutations.  In Proposition \ref{bifib} the generating tree connects the sequence to the Fibonacci numbers.

In Propositions \ref{perm1}, \ref{perm2}, and \ref{perm3} computation indicated that the sequences obtained for pattern-avoiding ascent sequences matched three different sequences that appear in the context of pattern-avoiding permutations, and that could be counted using the FINLABEL algorithm developed by Vatter.  That program derives a finitely labeled generating tree to describe permutations avoiding any set of permutation patterns that contain both the child of an increasing permutation and the child of a decreasing permutation.  Thus, it was reasonable to hunt for isomorphic generating trees for our sets of pattern-avoiding ascent sequences.

In short, a generating tree is a rooted labeled tree, such that the label of each vertex uniquely determines the labels of its children via \emph{succession rules}.  These rules, combined with a label for the root, then uniquely determine the tree and its labels.  If a generating tree for a family of objects $\mathcal{A}(n)$ is isomorphic to that of another family $\mathcal{B}(n)$, then we know $|\mathcal{A}(n)|=|\mathcal{B}(n)|$ for all $n$ (in addition to other structural similarities between the families).  Furthermore, if a generating tree uses only finitely many labels, then the Transfer Matrix Method of \cite[Section 4.7]{SEC1} immediately provides a (rational) generating function.

The first few levels of the tree of all ascent sequences is shown in Figure \ref{fig:treeall}.  The root is $0$ and the children of $x_1 \dotsm x_n$ are $x_1 \dotsm x_{n} x_{n+1}$ for each possible $x_{n+1}$ which yields an ascent sequence.   Thus we see the basic operation to be appending a letter to the end.  In each of Propositions \ref{perm1}, \ref{perm2}, and \ref{perm3}, the nodes of the tree will be labeled with ascent sequences, which helps to determine which additional letters may be appended while still avoiding the forbidden patterns.

\begin{figure}[ht]
\begin{center}
\begin{tikzpicture}[line cap=round,line join=round,>=triangle 45,x=1.0cm,y=1.0cm]
\clip(-6.46,2.48) rectangle (8.52,6.4);
\draw (0,6)-- (-4,5);
\draw (-4,5)-- (-5,4);
\draw (-4,5)-- (-2,4);
\draw (0,6)-- (4,5);
\draw (4,5)-- (1,4);
\draw (4,5)-- (3,4);
\draw (4,5)-- (5,4);
\draw (-5,4)-- (-5.76,3);
\draw (-5,4)-- (-4.44,3);
\draw (-2,4)-- (-3.2,3);
\draw (-2,4)-- (-2,3);
\draw (-2,4)-- (-1,3);
\draw (1,4)-- (0.16,3);
\draw (1,4)-- (0.7,3);
\draw (1,4)-- (1.48,3);
\draw (3,4)-- (2.46,3);
\draw (3,4)-- (3,3);
\draw (3,4)-- (3.58,3);
\draw (5,4)-- (4.76,3);
\draw (5,4)-- (5.48,3);
\draw (5,4)-- (6.3,3);
\draw (5,4)-- (7.16,3);
\begin{scriptsize}
\fill [color=black] (0,6) circle (1.5pt);
\draw[color=black] (0,6.2) node {$0$};
\fill [color=black] (-4,5) circle (1.5pt);
\draw[color=black] (-4,5.28) node {$00$};
\fill [color=black] (4,5) circle (1.5pt);
\draw[color=black] (4,5.28) node {$01$};
\fill [color=black] (-5,4) circle (1.5pt);
\draw[color=black] (-5,4.28) node {$000$};
\fill [color=black] (-2,4) circle (1.5pt);
\draw[color=black] (-2,4.28) node {$001$};
\fill [color=black] (1,4) circle (1.5pt);
\draw[color=black] (1,4.28) node {$010$};
\fill [color=black] (3,4) circle (1.5pt);
\draw[color=black] (3,4.28) node {$011$};
\fill [color=black] (5,4) circle (1.5pt);
\draw[color=black] (5,4.28) node {$012$};
\fill [color=black] (-5.76,3) circle (1.5pt);
\draw[color=black] (-5.76,2.8) node {$0000$};
\fill [color=black] (-4.44,3) circle (1.5pt);
\draw[color=black] (-4.44,2.8) node {$0001$};
\fill [color=black] (-3.2,3) circle (1.5pt);
\draw[color=black] (-3.2,2.8) node {$0010$};
\fill [color=black] (-2,3) circle (1.5pt);
\draw[color=black] (-2,2.8) node {$0011$};
\fill [color=black] (-1,3) circle (1.5pt);
\draw[color=black] (-1,2.8) node {$0012$};
\fill [color=black] (0.16,3) circle (1.5pt);
\draw[color=black] (0.1,2.8) node {$0100$};
\fill [color=black] (0.7,3) circle (1.5pt);
\draw[color=black] (0.76,2.8) node {$0101$};
\fill [color=black] (1.48,3) circle (1.5pt);
\draw[color=black] (1.48,2.8) node {$0102$};
\fill [color=black] (2.46,3) circle (1.5pt);
\draw[color=black] (2.3,2.8) node {$0110$};
\fill [color=black] (3,3) circle (1.5pt);
\draw[color=black] (3,2.8) node {$0111$};
\fill [color=black] (3.58,3) circle (1.5pt);
\draw[color=black] (3.63,2.8) node {$0112$};
\fill [color=black] (4.76,3) circle (1.5pt);
\draw[color=black] (4.76,2.8) node {$0120$};
\fill [color=black] (5.48,3) circle (1.5pt);
\draw[color=black] (5.48,2.8) node {$0121$};
\fill [color=black] (6.3,3) circle (1.5pt);
\draw[color=black] (6.3,2.8) node {$0122$};
\fill [color=black] (7.16,3) circle (1.5pt);
\draw[color=black] (7.16,2.8) node {$0123$};
\end{scriptsize}
\end{tikzpicture}
\end{center}
\caption{Tree of ascent sequences}
\label{fig:treeall}
\end{figure}
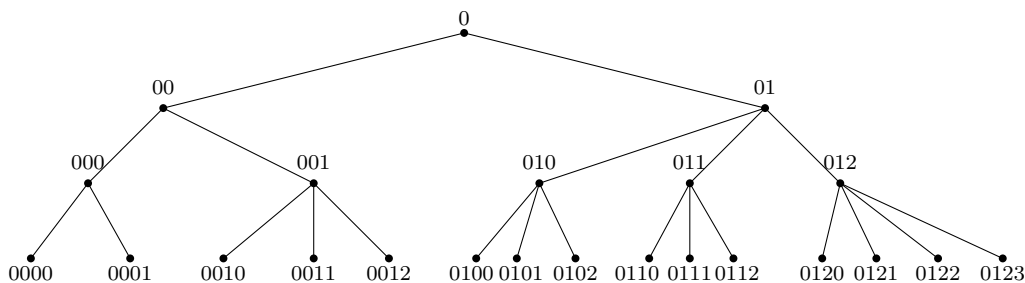

An important observation in the subsequent proofs is that consecutive copies of the same value beyond the first have no effect on pattern containment for patterns without repeated letters.  Formally, if $x$ can be decomposed into $x=X aa\dotsm a Y$, and $p$ is a pattern without consecutive repeated letters, then $x$ contains $p$ if and only if $XaY$ contains $p$.  It follows then that the same letters may be appended to $X aa\dotsm a Y$ and  $XaY$, and so these two ascent sequences have sets of children with the same multiset of labels.  We summarize these observations in a lemma:

\begin{lemma}\label{gtlem}
 For ascent sequences $ x = X aa\dotsm a Y$ and $x' = X a Y$ avoiding a pattern $p$ without consecutive repeated letters, the ascent sequence $xb$ is a child of $x$ if and only if $x'b$ is a child of $x'$ for any $b\geq 0$.   In terms of generating trees, $x$ and $x'$ may be given the same label since they produce isomorphic sets of children.
\end{lemma}

\begin{prop}\label{perm1}
$\mathrm{a}_{021,102}(n) = 3\cdot2^{n-1}-\binom{n+1}{2}-1$.  Equivalently, $\mathcal{A}_{021,102}(n)$ is in bijection with the set of permutations of length $n$ avoiding 123 and 3241.
\end{prop}

\begin{proof}
First note that avoiding 102 implies all ascent sequences will be RGFs by \cite{DS11}.  Therefore any letter appended to $x$ may be at most $\max(x)+1$.

Members of $\mathcal{A}_{021,102}(n)$ can be generated using a generating tree with precisely 5 labels: (0),(01), (010), (012), and (0120).  We have the following root label and succession rules.

\begin{itemize}
\item Root: (0)
\item Rules:
\begin{itemize}
\item $(0) \leadsto (0)(01)$
\item $(01) \leadsto (01)(010)(012)$
\item $(010) \leadsto (010)(010)$
\item $(012) \leadsto (0120)(012)(012)$
\item $(0120) \leadsto (0120)$
\end{itemize}
\end{itemize}

We now justify each succession rule in turn. 

$(0) \leadsto (0)(01)$: We first consider the root $0$, which we label as $(0)$.  This root has two children, $00$ and $01$.   By Lemma \ref{gtlem}, any all-zero ascent sequence will have an isomorphic set of children and so we label any all-zero ascent sequence with (0) (and only the all-zero sequences get this label).  Label the other child of an all-zero sequence (01).

$(01) \leadsto (01)(010)(012)$:   The first appearance of the (01) label is the ascent sequence 01, which has children 010, 011, and 012.  Again by Lemma \ref{gtlem} we may label 011 with (01) because of the repeated 1.  Note that any ascent sequence with form $0\dotsm0 1\dotsm1$ has label (01) and no others.  Label 010 with (010) and 012 with (012), and the rule follows.

$(010) \leadsto (010)(010)$:  Without considering pattern avoidance, the ascent sequence 010 has children 0100, 0101, and 0102.    However, 0102 contains the forbidden pattern 102, so this branch is pruned.  More broadly, if an ascent sequence contains a 010 pattern, no digits larger than 1 may be appended in order to avoid 102.  On the other hand, as many copies of 0 or 1 as we like may appear in the rest of the ascent sequence, so each ascent sequence starting with 010 has precisely two children: one where 0 is appended and one where 1 is appended since neither will produce a 021 nor 102.  Each of these children is given the label (010).  In this way, the label (010) is applied only to ascent sequences containing a 010 and consisting of  only 0s and 1s.

$(012) \leadsto (0120)(012)(012)$:  The first instance of the (012) is the ascent sequence 012, which has children 0120, 0121, 0122, and 0123 when ignoring pattern avoidance.  However, 0121 contains the forbidden pattern 021, so this branch is pruned. Label 0120 with (0120) and 0122 may be labeled with (012) by Lemma \ref{gtlem}.   The sequence 0123 may also be labeled (012) since it has an isomorphic set of children as 012: the only letters which may be added are 0, a copy of the last letter, or one more than the last letter.    It follows that any ascent sequence $x_1 \dotsm x_n$ labeled with (012) will be weakly increasing and its 021-avoiding children will be $x_1\dotsm x_n 0$, $x_1\dotsm x_n x_n$ and $x_1\dotsm x_n (x_{n}+1)$ which get labels (0120), (012) and (012) respectively.

$(0120) \leadsto (0120)$:  The children of 0120 are 01200, 01201, 01202, and 01203 when ignoring pattern avoidance.  All but 01200 contain some forbidden patterns, however.  More broadly, if an ascent sequence in this tree contains a $0120$ pattern then that pattern must be a literal 0120 and so only a 0 may be appended.  Since this child also contains a 0120 we give it the label (0120) as well.  


Finally, we use the transfer matrix method to obtain a closed form for the generating function $f_{021,102}(z) = \sum_{n \geq 0} \mathrm{a}_{021,102}(n) z^n$.  We have 5 labels in our generating tree, and we make a $6 \times 6$ production matrix $P$.  We associate columns 2 through 6 and rows 2 through 6 with the labels (0), (01), (010), (012), and (0120) respectively.  For $2 \leq i, j, \leq 6$, let $P_{i,j}$ be the number of nodes of label $j$ that are children of a node of label $i$.  Also, we let the first row have a 1 in column 2 to account for the root, (0).  The production matrix that results from our generating tree rules is $$P =\left[\begin{array}{cccccc}0&1&0&0&0&0\\0&1&1&0&0&0\\0&0&1&1&1&0\\0&0&0&2&0&0\\0&0&0&0&2&1\\0&0&0&0&0&1 \end{array}\right].$$  By construction the $j$th entry in the first row of $P^n$ is the number of nodes of type $j$ at the $n$th level of the generating tree, so $\mathrm{a}_{012,102}(n)= (P^n)_{1,2}+(P^n)_{1,3}+(P^n)_{1,4}+(P^n)_{1,5}+(P^n)_{1,6}$. It follows that $f_{021,102}(z) = u (I-zP)^{-1} e$, where $u$ is the row vector $u=\left[\begin{array}{cccccc}1&0&0&0&0&0\end{array}\right]$ and $e$ is a column vector consisting of six 1s.  This calculation yields the closed form $f_{021,102}(z) = \frac{z^4-3z^3+6z^2-4z+1}{(z-1)^3(2z-1)}$, from which we deduce the desired enumeration for $n \geq 1$ by the standard methods.

Using the Maple package \texttt{FINLABEL} in \cite{V06} we see that permutations avoiding $\{123,3241\}$ admit an 8-label generating tree, which has the same enumeration.  Although we have showed the cardinality of the ascent sequence set and the permutation set are the same, it remains an open problem to find a direct combinatorial bijection between them since the generating trees have differing numbers of labels.
\end{proof}

In the next two propositions we will employ similar methods and omit some of the details.  Unlike the generating tree of Proposition \ref{perm1}, the generating trees in Propositions \ref{perm2} and \ref{perm3} are isomorphic to the corresponding trees for pattern-avoiding permutations and so the proofs may be translated into bijections between the two sets.



\begin{prop}\label{perm2}
$\mathrm{a}_{102,120}(n) = (n-1)\,2^{n-2}+1$.  Equivalently, $\mathcal{A}_{102,120}(n)$ is in bijection with the set of permutations of length $n$ avoiding 132 and 4312.
\end{prop}

By Lemma \ref{RGFlemma}, since 102 is a subpattern of 01012, $\mathcal{A}_{102,120}(n)$ is in bijection with set partitions avoiding the patterns 1213 and 1231.   In Example 4.16 of \cite{JMS13}, Jelinek, Mansour, and Shattuck provide an algebraic proof of Proposition \ref{perm2} in this context of set partitions.  While the result is known, this is the first proof that can be considered combinatorial.

\begin{proof}
As noted in the preceding paragraph, avoiding 102 implies all ascent sequences are RGFs by Lemma \ref{RGFlemma}.  Therefore any letter appended to $x$ may be at most $\max(x)+1$.

Members of $\mathcal{A}_{102,120}(n)$ can be generated using a generating tree with precisely 3 labels: (0),(01), and (010).   We have the following root label and succession rules.

\begin{itemize}
\item Root: (0)
\item Rules:
\begin{itemize}
\item $(0) \leadsto (0)(01)$
\item $(01) \leadsto (01)(01)(010)$
\item $(010) \leadsto (010)(010)$
\end{itemize}
\end{itemize}

In this case we can state simply how labels are applied to ascent sequences.   Any all-zero sequence gets the label (0) and any other weakly increasing sequence will have the label (01).  All others, which each contain a 010 pattern, get label (010).  We now justify each succession rule in turn. 

$(0) \leadsto (0)(01)$: This is by the same argument as in the proof of Proposition \ref{perm1}. 

$(01) \leadsto (01)(01)(010)$:  The first appearance of the label (01) is the ascent sequence 01.  Without considering pattern-avoidance, 01 has children 010, 011, and 012.  By Lemma \ref{gtlem} we may label 011 as (01) as well, and we choose to label 010 with (010).   Further, because of the ascent 12 in 012, we may have no more copies of 0 in the rest of the ascent sequence (lest we have a copy of $120$), and so the minimum for the next appended letter(s) increases to 1.  Therefore $012$ may be labeled the same as 011, namely (01) without loss of significant information.  Following this same logic, any weakly increasing sequence gets labeled with a (01) and is the child of a node with label either (0) or (01). 

$(010) \leadsto (010)(010)$:  Ignoring pattern avoidance, the children of 010 are 0100, 0101, and 0102.  However, $0102$ contains the forbidden pattern 102 so that branch is pruned.  Lemma \ref{gtlem} implies 0100 can have the same label as 010.    For 0101, the second 1 is also irrelevant to formation of new copies of 102 and 120 since if that second 1 were part of a forbidden pattern then the first 1 would also be part of a forbidden pattern.  Therefore 0101 has the same options for children as 010.  More broadly, if $x$ has label (010), then $x$ contains a 010 pattern at $x_i x_j x_k$ and so the only children which avoid 102 and 120 are additional copies of the values $x_i$ or $x_j$ and so should also be labeled (010) by the same reasoning.

Via the transfer matrix method, these rules give the ordinary generating function $\frac{x^3-5x^2+4z-1}{(x-1)(2x-1)^2}$, from which we deduce the desired enumeration for $n \geq 1$.

Using the Maple package \texttt{FINLABEL} \cite{V06} we see that permutations avoiding $\{132, 4312\}$, or equivalently, their inverses $\{132,3421\}$ admit an isomorphic generating tree, and thus have the same enumeration.  Details of the implied bijection are left to the reader.
\end{proof}

\begin{prop}\label{perm3}
$\mathcal{A}_{101,120}(n)$ is in bijection with the set of permutations of length $n$ avoiding 231 and 4123. 
\end{prop}

\begin{proof}
First note that avoiding 101 implies all ascent sequences will be RGFs by \cite{DS11}.  Therefore any letter appended to $x$ may be at most $\max(x)+1$.

Members of $\mathcal{A}_{101,120}(n)$ can be generated using a generating tree with precisely 4 labels: (0),(01),(010), and (0102).  We have the following root label and succession rules.

\begin{itemize}
\item Root: (0)
\item Rules:
\begin{itemize}
\item $(0) \leadsto (0)(01)$
\item $(01) \leadsto (01)(01)(010)$
\item $(010) \leadsto (010)(0102)$
\item $(0102) \leadsto (01)(0102)$
\end{itemize}
\end{itemize}

We now justify each succession rule in turn.

$(0) \leadsto (0)(01)$:  This reasoning is the same as in the previous two propositions.  The ascent sequences with label (0) are exactly the all-zero sequences.

$(01) \leadsto (01)(01)(010)$:   The ascent sequence 01 has children 010, 011, and 012.  The first of these gets label (010).  As we have seen previously, 011 may be labeled as (01) by Lemma \ref{gtlem}.    As in the proof of Proposition \ref{perm2}, avoiding 120 means that 012 may be labeled with (01) since the minimum for appended letters is raised.  More broadly, if $x=x_1\dots x_n$ has label (01), let $m$ be the first letter of the last ascent in $x$ or $0$ if there is no ascent.  Then the only children are $xm$, $x x_n$, and $x (x_{n}+1)$, which get labels (010), (01), and (01) respectively.

$(010) \leadsto (010)(0102)$:   The ascent sequence 010 has children 0100, 0101, and 0102, but 0101 has the forbidden pattern 101 and is therefore pruned.  Lemma \ref{gtlem} implies 0100 can be labeled as its parent (010), and 0102 gets the new label (0102).  More broadly any ascent sequence $x$ with label (010) allows for two children: repeating the last letter or appending a letter larger than all letters in $x$.  These children get labels (010) and (0102) respectively.

$(0102) \leadsto (01)(0102)$: The ascent sequence 0102 has children 01020, 01021, 01022, and 01023, but the first two of these contain forbidden patterns.  Lemma \ref{gtlem} implies 01022 may be labeled as its parent. For 01023, the first three letters become irrelevant because no more copies of 0 or 1 may appear lest we create a 120 pattern.  Therefore the minimum for new letters  is raised and only the subsequence 23 matters for the sake of pattern avoidance.  Thus we label 01023 according to this pattern, (01).  More broadly any ascent sequence $x$ with label (0102) allows for two children: repeating the last letter or appending $\max(x)+1$.  These children get labels (0102) and (01) respectively.

Via the transfer matrix method, these rules give the ordinary generating function $\frac{(x-1)^3}{3x^3-5x^2+4x-1}$.  Using the Maple package \texttt{FINLABEL} in \cite{V06} we see that permutations avoiding $\{231,4123\}$ admit an isomorphic generating tree, and thus have the same enumeration.
\end{proof}

We close this section with an additional generating tree argument, although of a very different structure than the preceding propositions.



\begin{prop}\label{bifib}
$\mathrm{a}_{101,110}(n)=F_{2n-1}$
\end{prop}

This result first appeared as the $x=q=1$ case of Proposition 3.6 in \cite{MS13}.  There, Mansour and Shattuck track ascent sequences according to two statistics based on indices of left-to-right maxima and the value of the largest letter.  Our methods are similar, but are included to point out an interesting combinatorial connection to a refinement of the bisection of the Fibonacci numbers in the OEIS \cite{OEIS}.

\begin{proof}
We will construct a generating tree more typical of those in the literature: each ascent sequences is labeled by its number of children.  Thus a node with label $(k)$ has $k$ children.  In the end we will need to use infinitely many labels, but still be able to arrive at a simple formula.

First note that avoiding 101 implies all ascent sequences will be RGFs by Lemma \ref{RGFlemma}.  Therefore any letter appended to $x$ may be at most $\max(x)+1$.  

We see that the ascent sequence 0 has two children that avoid our given patterns, namely 00 and 01.  00 has 2 children that avoid our given patterns, 000 and 001.  01 has 3 children that avoid our given patterns, 010, 011, and 012.

In general, suppose that $x=x_1\dotsm x_n$ is in $\mathcal{A}_{101,110}(n)$ and has $k$ children that avoid our given patterns.  There are 3 cases to consider for a child $xz$ by appending digit $z$ to the end of $x$: (a) $z>x_n$, (b) $z=x_n$, or (c) $z<x_n$.  

If $z>x_n$, there is precisely one choice for $z$, $z = max(x)+1$, since $x$ is an RGF.  Further, this child $xz$ has one more child than $x$ does, since any digit that could have been appended to $x$ can still be appended to $xz$ without creating a forbidden pattern, and we could also append $z+1$.  Therefore if $x$ has label $(k)$, then this particular child has label $(k+1)$.

If $z=x_n$, then $z$ is determined uniquely.  As for the children of $xz$, we cannot use any digit less than $z$ later in the construction of the ascent sequence (lest we form a 110 pattern), so $xz$ has precisely two children, $xzz$ and $xz(\max(x)+1)$, and should receive the label $(2)$.

If $z<x_n$, 
then appending a letter to $xz$ less than $z$ creates a 110 pattern, and appending a letter with value between $z$ and $\max(x)+1$ creates a 101 pattern since $x$ is an RGF.   Therefore $xz$ has precisely two children, $xzz$ and $xz(\max(x)+1)$, and so has the label $(2)$.  These nodes labeled $(2)$ make up the remaining $k-2$ children of a node labeled $(k)$.

The above arguments yield the following root and succession rules:
\begin{itemize}
\item Root: (2)
\item Rules:
\begin{itemize}
\item $(2) \leadsto (2)(3)$
\item $(3) \leadsto (2)(2)(4)$
\item $(4) \leadsto (2)(2)(2)(5)$
\item $(k) \leadsto (2)^{k-1}(k+1)$
\end{itemize}
\end{itemize}

  We intend to show that the number of nodes on level $n$  (i.e., $\mathrm{a}_{101,110}(n)$) is given by $F_{2n-1}$, where $F_1=1$ and $F_2=1$.  

We will show a stronger result.  Let $b(n,k)$ be the number of nodes on level $n$ with label $k$.  Then clearly the number of nodes on level $n$ is given by $\sum_{k\geq 2} b(n,k)$.  We will show that:
\begin{equation}\label{eqn:fibtri}
 b(n,k) = \begin{cases}
    F_{2n-2k+2} & 2\leq k \leq n \\
   1 & k=n+1 \\
   0 & k<2 \textrm{ or } k>n+1 \\
\end{cases}
\end{equation}
Values of $b(n,k)$ for small $n$ and $k$ appear in Table \ref{tab:fibtri}.

\begin{table}
\begin{center}
\begin{tabular}{c|ccccc || c}
$n$~\textbackslash ~$k$ & 2 & 3& 4& 5& 6 & Row sum\\
\hline
1 & 1 & 0 & 0 & 0 & 0 & 1\\
2 & 1 & 1 & 0 & 0 & 0 & 2\\
3 & 3 & 1 & 1 & 0 & 0 & 5\\
4 & 8 & 3 & 1 & 1 & 0 & 13\\
5 & 21 & 8 & 3 & 1 & 1 & 34\\
\end{tabular}
\end{center}
\caption{Values of $b(n,k)$, the number of nodes of the generating tree for $\mathcal{A}_{101, 110}$ on level $n$ with label $(k)$.}
\label{tab:fibtri}
\end{table}

Our result that $\mathrm{a}_{101,110}(n)=F_{2n-1}$ follows from equation \eqref{eqn:fibtri} and the easily-proven  identity
\begin{equation}\label{eqn:fibid}
  F_{2n-1} = F_{2n-2} + F_{2n-4} + F_{2n-6} + \dotsm F_{2} + 1.
\end{equation}

Computation verifies the values for $b(n,k)$ for small $n$ and $k$.  Level 1 has only the node (2), corresponding to the ascent sequence $0$, and so $b(1,2)=1$ and $b(1,k)=0$ for any $k\neq 2$.   Level 2 corresponds to the ascent sequences 00 and 01, whose corresponding nodes are labeled (2) and (3) respectively and thus $b(2,2)=F_2=1$ and $b(2,3)=1$.


Observe that the succession rules imply that there is a bijection between nodes on level $n-1$ with label $k-1$ and nodes on level $n$ with label $k$ for $k \geq 3$.  Therefore $b(n,k) = b(n-1,k-1)$ for $k\geq 3$.  Iterating this recurrence relation implies $b(n,k) = b(n-k+2,2)$.  Therefore it remains to show $b(n,2) = F_{2n-2}$.  

We will proceed by induction: assume that $b(n,2)=F_{2n-2}$ and consider $b(n+1,2)$.  Since each node is labeled with its number of children, the sum of the labels on level $n$ equals the number of nodes on level $n+1$.  The succession rules imply each node with label $(k)$ produces $k-1$ nodes labeled (2), and so 
\begin{equation}
 \begin{split}
  b(n+1,2)&=\sum_{k=2}^{n+1} (k-1)\,b(n,k) \\
                &= \sum_{k=2}^{n+1} (k-1)\,b(n-k+2,2) \\
 \end{split}
\end{equation}
Apply the induction hypothesis to see
 \begin{equation}
 \begin{split}
  b(n+1,2) &= \sum_{k=2}^{n} (k-1)\,F_{2n-2k+2} + n\cdot 1 \\
                 &= 1\cdot F_{2n-2} + 2\cdot F_{2n-4} + 3\cdot F_{2n-6} + \dotsm (n-1) \cdot F_{2} + n \\
                 &= F_{2n}
 \end{split}
\end{equation}
where the last equality can be proven via induction and the identity $F_{2n+1} = 1+ \sum_{i=1}^{n} F_{2i}$.   
Alternately, let $G(x) = \sum_{n\geq 1} F_{2n}x^n$.  Then the obvious order 2 recurrence for $F_{2n}$ shows  $G(x)= x/(1-3x+x^2)$.  This generating function also satisfies $G(x) = x/(1-x)^2 \cdot (1+G(x))$ which implies the identity involving the convolution of the natural numbers with $F_{2n}$.
 Hence we have completed the proof of \eqref{eqn:fibtri}. 
\end{proof}

It should be noted that the triangle of values $b(n,k)$ appears in the OEIS \cite{OEIS} as A121461, where it is linked to nondecreasing Dyck paths and directed column-convex polyominoes based on work by Barcucci et al \cite{B97, B93}.

Those interested in other pairs of patterns $(\sigma, \tau)$ for which  $a_{\sigma,\tau}(n) = F_{2n-1}$ are directed to Proposition 3.3 of \cite{MS15}.


Generating trees also provide the crux of a proof for the following proposition.

\begin{prop}\label{binomcat}
$\mathrm{a}_{201,210}(n) = \sum_{k=0}^{n-1}\binom{n-1}{k}C_k$
\end{prop}

We defer the proof itself for a separate paper \cite{P15}, however, as it is signficantly more complicated than the arguments above.

\section{Conclusions and Future Directions}

In this paper we demonstrated multiple instances of pattern-avoiding ascent sequences that yield classical combinatorial avoidance sequences.  The ascent sequence interpretation of each sequence is new to the literature.  At this point, it should be clear that pattern-avoiding ascent sequences are in bijection with a number of other combinatorial structures (including Dyck paths, permutations, and set partitions), and that their avoidance sequences include a number of well-known classical sequences.  We also illustrated several instances of superfluous patterns, although make no claim that we have described all such instances.  To thie end we ask two questions:

\begin{enumerate}
  \item[(1)] What conditions for patterns $p$ and $q$ imply that $|\mathcal{A}_{p,q}(n)| = |\mathcal{A}_{p}(n)|$?
  \item[(2)] What conditions for patterns $p$, $q$, and $r$ imply that $|\mathcal{A}_{p,q}(n)| = |\mathcal{A}_{r}(n)|$?
\end{enumerate}

We have shown that generating trees are useful tools for studying several of these pattern-avoiding sets (see Propositions \ref{perm1}, \ref{perm2}, \ref{perm3}, \ref{bifib}).  In the case of pattern-avoidance in permutations, several other tools have been developed to automate enumeration of avoidance classes.  In a forthcoming paper we adapt one of these other tools, enumeration schemes, to better understand additional sets of pattern-avoiding ascent sequences.

\end{document}